\documentclass[12pt]{amsart}

\usepackage{stackrel}
\usepackage{amssymb}
\usepackage{amsmath}
\usepackage{amsthm}
\usepackage[pdftex]{graphicx}
\usepackage{mathtools}
\usepackage{pdfsync}
\usepackage{graphicx}
\usepackage{color}
\usepackage{tikz-cd}
\usepackage{tikz}
\usetikzlibrary{positioning}
\usetikzlibrary{decorations.text}
\usetikzlibrary{decorations.pathmorphing}
\usetikzlibrary{matrix}

\usepackage[all]{xy}

 \newtheorem{thm}{Theorem}[section]
 \newtheorem{cor}[thm]{Corollary}
 \newtheorem{lem}[thm]{Lemma}
 \newtheorem{prop}[thm]{Proposition}
 \theoremstyle{definition}
 \newtheorem{defn}[thm]{Definition}
 \newtheorem{rem}[thm]{Remark}
 \numberwithin{equation}{section}
 \newtheorem{eg}{Example}[section]

  \newcommand{\st}{\rm{such\  that\ }}
  \newcommand{\rk}{{\rm rank\ }}
  \newcommand{\iso}{\simeq}

     \newcommand{\Aff}{\text{Aff}}
\newcommand{\SP}{{\rm {Sp}\,}}
\DeclarePairedDelimiterX\Set[2]{\lbrace}{\rbrace}
 { #1 \,\delimsize|\,\mathopen{} #2 }
\newcommand{\Inf}{\rm{Inf}\,}

\def\({\left(}
\def\){\right)}

\def\st{such that }

 \newcommand{\R}{\mathbb{R}}
  
 \newcommand{\Z}{\mathbb{Z}}
  \newcommand{\N}{\mathbb{N}}
  \newcommand{\Q}{\mathbb{Q}}

\begin{document}
\title[]
 {Orbit equivalence of Cantor minimal systems and their continuous spectra}

\author{T. Giordano}\let\thefootnote\relax\footnote{The three authors were partially supported by NSERC operating grants.}
\address{Department of Mathematics and Statistics,  University of Ottawa, K1N6N5, ON   Canada}
\email{giordano@uottawa.ca}
\author{D. Handelman}
\address{Department of Mathematics and Statistics,  University of Ottawa, K1N6N5, ON   Canada}
\email{dehsg@uottawa.ca} 
\author{M. Hosseini}
\address{School of Mathematics, Institute for Research in Fundamental Sciences (IPM), P. O. Box 19395-5746, Tehran, Iran}
\email{maryhoseini@ipm.ir}


\maketitle
\begin{abstract}
To any continuous eigenvalue of a Cantor minimal system $(X,\,T)$, we associate an element of the dimension group $K^0(X,\,T)$ associated to $(X,\,T)$. We   introduce and study  the concept of {\it irrational miscibility\/} of a dimension group. The main property of these dimension groups is the absence of   irrational values in the additive group of continuous spectrum   of their realizations by Cantor minimal systems.  The strong orbit equivalence (respectively orbit equivalence) class of a Cantor minimal system associated to an irrationally miscible  dimension group $(G,\,u)$ (resp. with trivial infinitesimal subgroup) with trivial rational subgroup, have no non-trivial  continuous eigenvalues.

\end{abstract}

\section{Introduction}
The study of orbit equivalence in measurable dynamics was initiated by H. Dye \cite{Dye}, who proved in particular that any two ergodic probability measure preserving transformations are orbit equivalent. Therefore the spectrum of an ergodic probability measure preserving systems is independent of its orbit equivalence class. 
\medskip

In topological dynamics, the study of orbit equivalence was initiated by M. Boyle \cite{boyle} and was developed for Cantor minimal systems  by T. Giordano, I. Putnam and C. Skau in \cite{GPS} where the notion of strong orbit equivalence was also introduced; the two notions of orbit equivalence and strong orbit equivalence were characterized using two dimension groups $K^0(X,\,T)$ and $K^0_m(X,\,T)$ naturally associated to a Cantor minimal system $(X,\,T)$.  Since then, the study of topological orbit equivalence has been very active (see for example \cite{DSH}, \cite{GPMS1}, \cite{GPMS},  \cite{Ormes1}, \cite{Ormes}). 

\medskip
The  topological version of Dye's theorem  proved in \cite{GPS}, states  that a  uniquely ergodic Cantor minimal
system is orbit equivalent either to an {\it odometer} or to a {\it
Denjoy system}
\cite{GPS}. Odometers are
rotations on  a Cantor set and  therefore {\it automorphic\/}
systems (that is,  minimal isometries on  compact topological groups). Denjoy
systems (in special cases they are known as Sturmian systems), are
almost  one to one extensions of
irrational
rotations on the unit circle.
\medskip

For a non-uniquely ergodic Cantor minimal system, $(X,\,T)$,  the situation is
 more complicated as not every such system is 
orbit equivalent to an almost automorphic system.
Related to this phenomenon is the nature of the spectra
$${\rm Sp}(T)=\{\lambda\in {\mathbb S}^1:\ \exists \ f\in C(X,\,\mathbb C);\ f\circ T=\lambda f\}$$
(in the literature often
termed  the {\it continuous spectrum,} a standard object in operator
theory meaning something completely different).

 Recall (see \cite{GPS}, section 2.3) that for a Cantor minimal system $(X,\,T)$, its associated  {\it pointed dimension group,} $(K^0(X,\,T),\,K^0_+(X,\,T),\,[1_X])$,  is a complete invariant for  strong orbit equivalence.
\medskip

Information about the relationship between  (strong) orbit equivalence
class of
Cantor minimal systems and their spectra can be derived from the value
group of their invariant measures on clopen sets,
and also from the {\it rational subgroup,} $\mathbb Q(K^0(X,\,T),\,[1_X])$,  of
their dimension
group. In \cite{Ormes1},  Ormes  proved that for any $\lambda=e^{2\pi
i\theta}$ in the spectrum of a Cantor minimal system $(X,\,T)$
there exists $f\in C(X,\,\mathbb Z)$ \st $f-\theta 1_X$ is a real
coboundary. From Theorem \ref{clopen}, the
function $f$ can be chosen to be the indicator function of a clopen set.
Consequently,  for any $0\leq \theta <1$, such that $e^{2\pi i\theta}\in {\rm Sp}(T)$, there exists a clopen set whose measure is $\theta$ with respect to all invariant probability measures on $X$. Equivalently,
$$
\SP(T)\subseteq\bigcap_{\tau\in S(K^0(X,\,T),\, [1_X])}\tau(K^0(X,\,T)),
$$
where $S(K^0(X,\,T),\, [1_X])$ denotes the set of normalized traces (states) on $K^0(X,\,T)$. This result has been recently   stated in \cite[Theorem 1.]{cortez} as well.
(The reverse inclusion fails in general, as there are many
uniquely ergodic weakly mixing systems.)

Therefore, if a Cantor minimal system has an   eigenvalue $\exp(2\pi
i\theta)$ for some irrational $\theta$, there will be an irrational number
in the set of values of its traces. But not all the irrational numbers
which  appear in   the  intersection are eigenvalues for
$(X,\,T)$. If the intersection contains only rational numbers (the pointed
dimension group is then called {\it irrationally mixing\/}), there will be
no irrational eigenvalues (the converse fails), and when this occurs, it
is relatively easy to decide when all systems orbit or strongly orbit
equivalent to it are weakly mixing. This applies to various classes of
simple dimension groups, for example, if $G$ has  $n>1$
pure traces, and either $\rk(G/\Inf(G))=n$ or if $\rk G = n+1$ and $G$ is
finitely generated (as an abelian group), then all realizations of $G$ by
Cantor minimal systems cannot be even
Kakutani equivalent to a Cantor minimal system with irrational
eigenvalues. There are large classes of such dimension groups already in
the literature.

The rational subgroup of a dimension group corresponds to  the rational part of the
spectrum. Indeed, it was proved in \cite{Ormes} that $exp(2\pi i1/p)\in
\SP(T)$ if and only if $1/p\in\mathbb Q(K^0(X,\,T),[1_X])$. When the system is weakly
mixing, the rational subgroup is
isomorphic to $\Z$ (a rank one cyclic subgroup).  We describe   in Proposition \ref{weakmixing}  (resp. \ref{last}),  strong orbit equivalent (resp. orbit equivalent)  classes of Cantor minimal systems which are   weakly mixing.
\medskip

Let $(B,\,V)$ be a Bratteli diagram whose associated dimension group has a trivial rational subgroup. In Theorem \ref{propo}, we show that there exists a telescoping $\tilde{B}$ of $B$ and a proper ordering on $\tilde{B}$ whose associated Bratteli-Veshik system is weakly mixing.

\subsection{Topological dynamical systems}

A {\it topological dynamical system\/} is a pair $(X,\,T)$ wherein $X$ is a
compact metric space and $T$  is a self-homeomorphism. A topological
dynamical system is said to be {\it topologically transitive\/} if there exists a point
in $X$ with dense orbit;  it is  {\it minimal\/}
when every point has  dense orbit.
Two dynamical systems $(X,\,T)$ and $(Y,S)$ are {\it conjugate\/} if there exists a homeomorphism $h: X\rightarrow Y$ such that
$h\circ T=S\circ h$. If $h$ is merely a continuous surjection then $(Y,\,S)$ is a {\it factor of $(X,\,T)$\/} and the
function $h$ is called a {\it factor map\/}, and $(X,\,T)$  an
{\it  extension of $(Y,\,S)$.} An {\it almost one to one extension\/} of a
minimal system is an extension $h:X\rightarrow Y$ for which  the set of
points with unique  pre-image is residual in $Y$; this  is equivalent to
there being at least one point with  unique pre-image.

By a {\it Cantor system,} we mean a topological dynamical system on a
totally disconnected metric space without isolated points. As an example one may consider a {\it subshift} system \cite{glasner}; let $A=\{1,\,\dots,\,\ell\}$ and $\Omega=A^{\Z}$ of all bi-infinite sequences on the finite alphabet $A$ with the product topology.  Let $(\Omega,\,T)$ be the topological dynamical system defined by the  left shift map. If $X$ is any $T$-invariant closed subset of $\Omega$,  $(X,T)$ is
called a {\it subshift}.  {\it Substitutions\/} are   Cantor minimal  subshifts.

{\it Kronecker  systems }are minimal systems on a compact metric group; examples
include  irrational rotations on the unit circle and {\it odometers\/} on a shift space. The
odometer associated to the sequence $(a_1, a_2,\dots)$, $a_i\geq 2$ is the pair
$(X,\,T)$ with
  $X=\prod_{n=1}^\infty\{0,1,\dots,a_n-1\}$ endowed with the product
topology and $T$ is defined by addition of $(1, 0, 0, \dots)$ with
carriage to the right. An almost one to one extension of a Kronecker
system is called an {\it  almost automorphic\/} system.

\bigskip

A complex number $\lambda=e^{2\pi i\theta}$ is
a {\it continuous  eigenvalue\/} for $(X,\,T)$ if there exists a continuous function
$f: X\rightarrow\mathbb S^1$ such that $f\circ T=\lambda f$;  then $f$ is said to
be the corresponding {\it eigenfunction\/} for the system. An
eigenfunction is a factor map from  $X$ into the unit circle; this yields
a factor map onto the unit circle iff $\theta$ is an irrational number.
The set of all eigenvalues of $(X,\,T)$ is called its {\it spectrum,}
denoted $\SP(T)$.  This is sometimes called the {\it continuous spectrum,}
but this term has been pre-empted by operator theory. Every topological
dynamical system has $\lambda=1$ in its spectrum, and if $(X,\,T)$ is a minimal system and ${\rm Sp}(T)=\{1\},$   then $(X,\,T)$ is
weakly mixing.

\begin{defn}\label{oe}
Two Cantor minimal system, $(X,\,T)$ and $(Y,\,S)$,  are {\it orbit
equivalent\/} if there exists a homeomorphism $F:X\rightarrow Y$ such that
$F({\mathcal O}_T(x))={\mathcal O}_S(F(x))$ for all $x\in X$.
\end{defn}

When  two systems are orbit equivalent, it follows from the definitions
that there exists an integer-valued function $n:X\rightarrow{ \mathbb Z}$
such that  for each point $x$ in $X$, $F(T(x))=S^{n(x)}(F(x))$, and
similarly, there exists a map $m: Y\rightarrow \mathbb Z$ such that
$F(T^{m(x)}(x))=S(F(x))$. The maps $m$ and $n$ are called  {\it orbit cocyles}.
Since the systems are minimal, the two cocycles are uniquely defined.
\
\begin{defn}\label{soe}
Two Cantor minimal systems are {\it strongly orbit equivalent\/} if there
exists an orbit equivalence $F$ such that the two orbit  cocycles $m$ and $n$
arising from $F$ have at most one point of discontinuity.
\end{defn}

\begin{defn}
Let $(X,\,T)$ be a Cantor minimal system and  $U$ be a clopen subset of
$X$. The {\it first return}  map $T_U:U\rightarrow U$ is defined by $
T_U(x)=T^{r_U(x)}(x)$ where , $$r_U(x)=\inf\{n\in\mathbb Z^+:\ \
T^n(x)\in U\}.
$$
The new Cantor minimal system, $(U,\,T_U)$, is called the {\it induced
system }of $(X,\,T)$ with respect to  $U$.
\end{defn}

\subsection{Ordered Bratteli diagrams}

A Bratteli diagram is an infinite directed graph which consists of a
vertex set, $V$, and an edge set, $E$, such that $V$, $E$ are   unions of countably
many non-empty finite sets,
$$
V=V_0\dot\cup V_1\dot\cup\cdots;\ \ \ V_0=\{v\}, \ \ \ \ E=E_1\dot\cup
E_2\dot\cup\cdots.
$$
There are two maps, the {\it range} and the {\it source}, $r,s:
E\rightarrow V$, with $r(E_n)\subset V_n$ and $s(E_n)\subset V_{n-1}$. A
vertex $u\in V_n$  is {\it connected\/} to $v\in V_{n+1}$ if there is an
edge $e\in E$ such that $r(e)=u$ and $s(e)=v$.

So for any $n\in\mathbb N$, there is a $|V_n| \times |V_{n-1}|$  incidence matrix $A_n$. Each $a_{ij}$ thus counts the number of edges from
$v_j\in V_n$ to  $v_i\in V_{n+1}$. The diagram obtained by considering the
vertices of $V_n$  as the nodes arranged horizontally in the $n$th level
of a diagram. Then the $i$th node in the $n$th level will be
connected to the $j$th node  in the $n+1$st level by $a_{ij}$ edges.

Let $\{n_k\}_{k=0}^\infty$ be an increasing sequence of natural numbers
with $n_0=0$. We may telescope the diagram along this sequence, obtaining
a new one, by  taking  $V'$ and $E'$ to be the sets of vertices and edges
such that $V'_k=V_{n_k}$, and choosing the incidence matrix between two
consecutive levels $k'$ and $(k+1)'$ to be  $A'_k=A_{n_k}A_{n_k+1}\cdots
A_{n_{k+1}}$.
A Bratteli diagram is {\it simple\/} if it admits a telescoping such that
all the resulting incidence matrices have (strictly) positive entries. The
matrices $A_n$ yield a direct limit partially ordered abelian group, known
as the dimension group (associated to the Bratteli diagram); see below.

\begin{defn}
An {\it ordered Bratteli diagram,} $B(V,\,E,\,\geq)$, is a Bratteli
diagram $(V,\,E)$ with a partial ordering $\geq$ on its edges such that two edges
$e$ and $e'$ are comparable iff $r(e)=r(e')$;  In other words, each set 
$r^{-1}(v),$  $v\in V\setminus V_0,$ is linearly ordered. The edge
with the biggest (least) number in the ordering is called the {\it max
edge\/} ({\it min edge\/}).
\end{defn}
There is an induced ordering on any telescoped diagram. For any two
positive integers $l$ and $k$ with $k<l$, the set $E_{k+1}\circ
E_{k+2}\circ\cdots\circ E_l$ running from  $V_k$ to $V_l$ can be ordered
as follows: $(e_{k+1}, e_{k+2},\dots,e_l)>(f_{k+1}, f_{k+2},\dots,f_l)$
iff there exists some $i$ with $k+1\leq i\leq l$ such that for all
$i<j\leq l$, $e_j=f_j$ and  $e_i>f_i$.

Let $(V,\,E,\,\geq)$ be an ordered Bratteli diagram and $X_B$ denote the
associated infinite path space,
$$
X_B=\{(e_1, e_2,\cdots):\ \ e_i\in E_i,\ r(e_i)=s(e_{i+1});\ \
i=1,2,\dots\}.
$$
Two paths are {\it cofinal\/} if almost all  their edges agree.  The usual
compact topology on the path space $X_B$ has the set of  cylinder sets as
a basis, where cylinder sets are of the form,
$$
U(e_1, e_2,\dots,e_k)=\{(f_1, f_2, \dots)\in X_B:\ \ \ f_i=e_i,\ 1\leq
i\leq k\}.
$$
 Equipped by this topology, $X_B$ is a compact Hausdorff space with a
countable basis consisting of clopen sets; it  is called {\it Bratteli
compactum.} When $(V,\,E)$ is a simple Bratteli diagram, then $X_B$ has
no isolated points, so is a Cantor set. Let $X_B^{max}$ denote the set of
all those elements $(e_1, e_2, \dots)$  in $X_B$ such that each  $e_n$ is
a maximal edge, and define $X_B^{min}$ analogously.  An ordered Bratteli
diagram is  called {\it properly ordered\/} if $(V,\,E)$ is a simple
Bratteli diagram and $X_B^{max}$ and $X_B^{min}$ each contain only one
element; when this occurs, the maximal and minimal paths are  denoted
$x_{max}$ and $x_{min}$ respectively. For any Bratteli diagram, there
exists an ordering which makes it   properly ordered  \cite{HPS}.

Let $(V,\,E,\,\leq)$ be a simple properly ordered Bratteli diagram. The {\it Vershik\/
{\rm(or } adic\/{\rm)} map} is the (minimal) homeomorphism $\varphi_B:
X_B\rightarrow X_B$ wherein    $\varphi_B(x_{max})=x_{min}$, and for any other
point $(e_1, e_2,\dots)\neq x_{max}$,  the map sends the path to its successor \cite{HPS}; in particular,  let $k$ be the smallest number that
$e_k$ is not a max edge, let $f_k$ be the immediate successor of $e_k$,
and then
$\varphi_B(e_1, e_2,\dots)=(f_1, \dots, f_{k-1},f_k,e_{k+1},e_{k+2},\dots)$,
where $(f_1,\dots, f_{k-1})$ is the minimal edge in $E_1\circ
E_2\circ\cdots\circ E_{k-1}$ which has the same source as $f_k$.

Using Kakutani-Rokhlin partitions for  Cantor minimal systems, Herman, Putnam and Skau
proved the following result.

\begin{thm}{\cite{HPS}}
Let $(X,\,T)$ be a Cantor minimal system. Then $T$ is
topologically conjugate to a Vershik map on a Bratteli compactum $X_B$  of
a properly ordered Bratteli diagram $(V,\,E\,\leq)$. Furthermore, given
$x_0\in X$ we may choose the conjugating map $f:X\rightarrow X_B$ so that
$f(x_0)$ is the unique infinite maximal path in $(V,\,E\,\leq)$.
\end{thm}

\subsection{Dimension groups}
\begin{defn}
A {\it unital\/} (or {\it pointed\/}) partially ordered group is a triple
$(G,\,G_+,\,u)$  where
\begin{itemize}
\item $G$ is an abelian group,
\item $G_+$ is a subset of $G$ such that
$$
G_+\cap (-G_+)=\{0\},\ \ \ G_++G_+\subseteq G_+,\ \ \ G_+-G_+=G,
$$
\item $u$ is an element of $G_+$ such that for any $g\in G$ there exists
$n\in\mathbb Z^+$ such that $(nu-g)\in G_+$.
\end{itemize}
\end{defn}
A partially ordered group is a {\it dimension group} if
\begin{itemize}
\item $(G,G^+)$ is {\it unperforated\/}: if  $g\in G$ and $ng\in G^+$ for
some $n \in \mathbb N$, then $g\in G^+$.
\item $(G,\,G^+)$ satisfies the {\it Riesz interpolation property}: if
$g_1, g_2, h_1, h_2\in G$ and $g_i\leq h_j$, $i, j=1,2$, then there
exists a $g\in G$ such that $g_i\leq g\leq h_j$ for all $i, j$.
\end{itemize}

A subgroup $J$ of the dimension group $G$ is called an {\it order ideal} if $J=J^+-J^+$, where $J^+=J\cap G^+$ with the property that  $0\leq a\leq b\in J$ implies $a\in J$. A dimension group is called {\it simple} if it has no non-trivial order ideal. In a simple dimension group, any $u\in G^+\setminus \{0\}$ is an order unit \cite{HPS}.

\bigskip
A homomorphism $p:G\rightarrow \R$ is a  {\it state} (or {\it normalized  trace}) if $p(G^+)\geq 0$ and $p(u)=1$. The collection of all (normalized) traces on $G$   is a compact convex subset of $\mathbb R^G$  denoted by $S(G,\,u)$. 
\bigskip

Recall that an element $g$ of a partially ordered group $G$ (with an order unit $u$) is {\it infinitesimal\/} if $-\epsilon u\leq g\leq \epsilon u$, for all $0<\epsilon\in\Q$. The set of all infinitesimals of $G$ constitutes the {\it infinitesimal subgroup\/} of $G$, denoted $\Inf(G)$.

The {\it rational subgroup of a simple dimension group $(G,\,G_+)$  with  order unit $u$} is 
$$\Q(G,\,u):=\{g\in G:\ \exists p\in\Z,\  n\in \N \ {\rm such\ that}\  pg=nu\}.$$
\bigskip

Let $(G,u)$ be a simple noncyclic dimension group. Recall \cite{EHS} that  the double dual map  $g \mapsto \hat
g$, where $\hat{g}(\tau)=\tau(g)$, is an order preserving  affine representation, $G \rightarrow {\Aff} (S(G,u))$.
\bigskip

For $(G,\,u)$, a countable unperforated partially ordered abelian group with order unit, let $J(G,\,u)$ dente the subgroup of $G$ given by 
$$J(G,\,u)=\{g\in G:\ \hat{g}\ {\rm is\  constant}\}$$
and $I(G,\,u)$ be the subgroup of $\R$ defined as  $$I(G,\,u)=\bigcap_{\tau\in S(G,\,u)}\tau(G).$$ For any $g\in J(G,\,u)$ let  $\Phi(g)$ be the real number given by  $\hat{g}(\tau)$, for any (all) $\tau\in S(G,\,u)$. 
\bigskip

\smallskip

\begin{prop}\label{residual} Let $(G,u)$ be a countable unperforated partially ordered
abelian group with order unit. Then
 $$\{\tau\in S(G,u);\ \ker \tau= \Inf G\}$$ is a dense G${}_{\delta}$ in 
$S(G,u)$.
\end{prop}

\begin{proof}
For each $g \in G \setminus \Inf G$, define $U_g := \Set{\sigma \in
S(G,u)}{\sigma(g) \neq 0} = (\hat g^{-1}({0}))^c$. This is clearly open in
$S(G,u)$, and is also dense: take any element $\tau\in S(G,\,u)$ with $\tau(g)=0$, since $g \not\in \Inf G$, there exists $
\sigma \in S(G,u)$ such that  $\sigma(g)\neq 0$ \cite{GW}; then $\phi_n :=
(1/n)
\sigma + (1-1/n)\tau \to \tau$, and $\phi_n (g)
\neq 0$, so $\phi_n \in U_g$.

As $S(G,u)$ is compact,  $\cap_{g \in G \setminus {\rm Inf} G} U_g$ is a
dense
G${}_{\delta}$, and this is precisely the set of traces $\tau$ with $\ker\tau = {\Inf} G$.
\end{proof}

\bigskip

\begin{cor}\label{intersection}
 Let $(G,u)$ be a countable unperforated partially ordered abelian group
with order unit $u$. Then
$$\bigcap_{\tau \in S(G,u)} \tau(G) = \Set{\lambda \in \R}{\exists g \in G\
{ {\rm such\ that}\ } \ \hat g = \lambda \pmb 1}.
$$
\end{cor}

\begin{proof}
The right side is clearly contained in the left. Suppose $\alpha$ belongs
to the left side but not the right. By Proposition \ref{residual}, there exists a trace
$\tau \in S(G,u)$ such that
$\ker \tau = {\rm Inf} G$.
There exists $g \in G$ such that $\tau(g) = \alpha$. Since $\alpha$ does not
belong to the right side, $\hat g$ is not constant, so there exists
$\sigma \in S(G,u)$ such that  $\sigma(g) \neq \alpha$.

Let $F$ be the subfield of $\R$ generated by $\tau(G) \cup \sigma(G)$;
this is countable, so there exists $\beta $ in the open interval $(0,1)$
such that  $\{1,\beta\}$ is linearly independent over $F$. Set $\phi = \beta
\tau + (1-\beta)\sigma$. As $\alpha$ belongs to the left side, there
exists $h \in G$ such that  $\phi(h) = \alpha$. This yields $\alpha = \beta
\tau(h) + (1-\beta)\sigma (h)$, whence $(\tau (h) - \sigma(h))\beta +
(\sigma(h) - \alpha) = 0$. As $\sigma(h), \tau(h), $ and $\alpha =
\tau(g)$ belong to $F$, we deduce $\tau(h) = \sigma(h) = \alpha$.

Thus $\tau(g) = \tau(h)$, so $g - h \in {\rm Inf} G$. Hence $\sigma(g)=
\sigma(h)$, and the latter equals $\alpha$, contradicting $\sigma(g) \neq
\alpha$.
\end{proof}

\bigskip
\begin{cor}\label{exact}
If $(G,\,u)$ is a simple dimension group, then  we have the following short exact sequence:
$$0\longrightarrow \Inf(G)\longrightarrow J(G,\,u) \stackrel  {\Phi} {\longrightarrow}I(G,\,u) \longrightarrow 0.$$
\end{cor}
\begin{proof}
The proof is a straightforward consequence of Corollary \ref{intersection} and the definitions of $J(G,\,u)$, $I(G,\,u)$,  and $\Phi$.
\end{proof}
\bigskip

Let $(X,\,T)$ be a Cantor minimal system and $C(X,\,\mathbb Z)$ be the abelian group
of all continuous integer-valued functions on $X$.  Denote by
$\partial_TC(X,\,\mathbb Z)$  the subgroup of all elements, $g$,  in
$C(X,\mathbb Z)$ which can be represented in the form $g=f-f\circ T$ for some $f\in
C(X,\,\mathbb Z)$; an alternative notation is $g = (I-T)f$. Each element
of $\partial_TC(X,\,\mathbb Z)$ is called an {\it integer coboundary.}
Define $K^0(X,\,T)=C(X,\,\mathbb Z)/\partial_TC(X,\,\mathbb Z)$  and set
$K^0_+(X,\,T)$ to be the semigroup of equivalence classes of nonnegative
functions, That is,
$$K^0_+(X,\,T)=\{[f]:\ \  f\geq 0\}.$$ Denote by $[1_X]$,
the equivalence class of the constant function $1$ on $X$.
Then
$$
(K^0(X,\,T),\,K^0_+(X,\,T),\,[1_X])
$$
is a simple non-cyclic dimension group. Moreover, any simple non-cyclic dimension group, $G$, can be realized as $K^0(X,\,T)$ for some Cantor minimal system $(X,\,T)$ \cite[Theorem 6.2] {HPS}.  There is a bijective correspondence between $S(K^0(X,\,T),\,[1_X])$ and $\mathcal M_T(X)$, the space of all invariant measures on $(X,\,T)$. For a Cantor minimal system $(X,\,T)$, we have that
$
\{f:\ \int f\ d\mu=0\  \forall \ \mu\in {\mathcal
M}_T(X)\}
$
is a subgroup of $C(X,\,\Z)$ containing $\partial_TC(X,\,\Z)$. Then
$${\Inf}(K^0(X,\,T))=\{f:\ \int f\ d\mu=0\ \forall \ \mu\in {\mathcal
M}_T(X)\}/ \partial_TC(X,\,\mathbb
Z)$$ and  $$K^0_m(X,\,T)=C(X,\,\mathbb
Z)/\{f:\ \ \int f\ d\mu=0\ \ \forall \ \mu\in {\mathcal
M}_T(X)\}.
$$ Indeed,
$K^0_m(X,\,T)=K^0(X,\,T)/{\Inf}(K^0(X,\,T))$, and both $K^0(X,\,T)$ and $K^0_m(X,\,T)$ are  simple dimension groups (with the obvious orderings), the latter with  order unit $[1_X]$.

\begin{thm}{\cite{GPS}}\label{orbit}
Two Cantor minimal systems, $(X,\,T)$ and $(Y,\,S)$,  are orbit equivalent
iff $$(K^0_m(X,\,T), K^0_m(X,\,T)^+,[1_X])\iso(K^0_m(Y,\,S),\,K^0_m(Y,\,S)^+,\,[1_Y]).$$
\end{thm}

\begin{thm}\label{storbit}
{\cite{GPS}} Two Cantor minimal systems, $(X,\,T)$ and $(Y,\,S)$,  are
strongly orbit equivalent iff
$$
(K^0(X,\,T),
K^0(X,\,T)^+,[1_X])\iso(K^0(Y,\,S),\,K^0(Y,\,S)^+,\,[1_Y]).
$$
\end{thm}

\bigskip

\section{Spectra and real coboundaries}

Let$(X,\,T)$ be a  Cantor minimal system. If $E(X,\,T)$ denotes as  in \cite{durrand}, the subgroup of all real numbers $\theta$ that $\exp(2\pi i\theta)\in \SP (T)$, we construct in this section an embedding $\Theta$ of $E(X,\,T)$ into $K^0(X,\,T)$ (Corollary \ref{cor1}) and study properties of the image $\Theta(E(X,\,T))$. 
\bigskip

Recall that a continous function $f:X\rightarrow \R$ is a real $T$-coboundary if $f=g-g\circ T$ for some $g\in C(X,\,\R)$. Real co-boundaries were characterized in \cite{GH} by Gottschalk \& Hedlund and studied in particular in \cite{Ormes1} and \cite{Quas}. 

\bigskip

The following lemma is    well known, but as it plays an important role in Theorem \ref{clopen}, we include a proof.

\begin{lem}
Let $X$ be a totally disconnected compact
space, and suppose that $f:X\rightarrow \mathbb \mathbb \mathbb S^1$ is continuous.
For  every $\epsilon > 0$, there exists a continuous function $ F_\epsilon:
X\rightarrow [-\epsilon,\,1]$ such that
 $f = \exp( 2\pi i F_\epsilon)$.
\end{lem}
\begin{proof}
With suitable choice of analytic branches of the logarithm, for each $x \in X$, there exists a neighbourhood $U_x$ and a continuous function   $ E_x:
U_x\rightarrow [-\delta, 1+\eta]$, $\delta, \eta\geq 0$,  such that $f(y)=\exp(2\pi i  E_x(y))$,  for $y\in U_x$. Then 
compactness of $X$ yields a finite open covering by $\{{U_{x(j)}}\}_j$, and total
disconnectedness yields a finite disjoint clopen covering $\{V_j\}_{j=1}^n$
with $V_j \subseteq U_{x(j)}$. To finish the proof, set  $F_{\epsilon} = \sum_{j=1}^n \chi_{V_j}
E_{x(j)}|U_{x(j)}$.
{\rm such\ that}\
\end{proof}

\bigskip

Let $(X,\,T)$ be a Cantor minimal system; as in \cite{cortez}, denote  by $E(X,\,T)$ the subgroup of real numbers consisting  of all  $\theta$ such that $\exp(2\pi i\theta)\in \SP(T)$.  This is the additive  group of (continuous) eigenvalues.

\begin{thm}\label{clopen}
 Let $(X,T)$ be a  transitive topological dynamical system,
and suppose that  $X$ is totally disconnected.  Then $ \theta\in E(X,\,T)$, $0<\theta<1$,  if and only if  there exists a clopen set $U=U_\theta$ {\rm such\ that}\ $$1_{U_\theta}-\theta\cdot \pmb 1 \ {\rm  is\  a\  real\  coboundary}.$$ Moreover, for every
 $\mu\in {\mathcal M}(X,\,T)$, $\mu(U_\theta) = \theta$.

\end{thm}
\begin{proof}
Let $f$ be a continuous eigenfunction with respect to  $\exp (2\pi i \theta$. The function $|f|$ is continuous and  constant on orbits, at least one of
which
is dense.
By replacing $f$ by $f/|f|$ if necessary, we may assume that  $|f| = 1$.
Choose $0 < \epsilon < \min\{\theta,\, 1-\theta\}$ and write $f = \exp (2
\pi i
F_{\epsilon})$ by  the previous lemma. Then for all $x$ in $X$,
there exists an integer $k(x)$ \st $F_\epsilon(x) - F_\epsilon(Tx))
= -\theta + k(x)$.

We have $-\epsilon \leq F_\epsilon(T(x)) \leq 1$ and $\epsilon +
\theta < 1$. Thus
$$ -1< -1+( \theta-\epsilon)
 \leq k(x) \leq \epsilon + F_\epsilon(x) + \theta<2.
$$
But $k(x)$ is an integer-valued function, so $k(x)\in\{0,\,1\}$; in
addition, $k(x)=F_\epsilon(x) - F_\epsilon\circ T(x) + \theta$ is 
continuous. Hence $k = 1_U$ for a clopen subset $U$ of $X$.

For any invariant measure $\mu$
$$
\mu(U)=\int_X 1_U\, d\mu=\int_X k\,d\mu=\int_X (F_\epsilon -
F_\epsilon\circ T + \theta)d\mu=\theta.
$$

The converse is straightforward. If $1_U - \theta\cdot\pmb 1$ is a real
coboundary, 
that is an element of $(1-T)C(X,\mathbb R)$, then $\lambda := \exp 2 \pi i
\theta$ is an eigenvalue of $T$. Explicitly, if $1_U - \theta \cdot \pmb 1 =
(1-T)F$ where $F$ is real-valued and continuous, then we see
$$
f:= \exp 2 \pi i F = \exp 2 \pi i(F\circ T  + 1_U - \theta \cdot\pmb 1)
$$
$$= f \circ T \cdot 1 \cdot \lambda^{-1}. $$
Hence $f \circ T = \lambda f$.

In the case that $\theta = 1/q$, we take    $F = \sum_{j=0}^{q-1}
1_{T^j U} j/q$.
\end{proof}

\bigskip
\begin{rem}\label{affine}
Let $(G,\,u)=(K^0(X,\,T),\,[1_X])$.  With the notation of sections 2, 3, the conclusion of  Theorem \ref{clopen}  can be rephrased as
follows: If $ \theta\in E(X,\,T)\cap (0,\,1)$, then  there exists $g \in G^+$ \st  $\hat g = \theta \cdot \pmb 1$.\end{rem}

\bigskip

Let us say that a dynamical system has {\it sufficiently many measures\/} if every
nonempty open set has
nonzero measure for at least one invariant probability measure. Every
minimal system obviously has this property.

\begin{cor}\label{rational}
Let $(X,T)$ be a  topologically transitive  dynamical system
where $X$ is totally disconnected space, and there are  sufficiently many $T$-invariant measures.  If
$p$ and $q$ are relatively prime positive integers  \st  $\theta =
p/q\in E(X,\,T)\cap (0,\,1)$, then there exists a clopen set $U_\theta$ of $X$ \st $q[1_{U_\theta}] =
p[1_X]$ in $K^0(X,\,T)$.
\end{cor}

\begin{proof}
Suppose that  $\theta = p/q$ belongs to $E(X,\,T)\cap (0,\,1)$.  For $0<\epsilon<1$, set $F:=F_\epsilon$.
First assume that $p=1$. Let $U=U_\theta$  be given by Theorem \ref{clopen}.
 We claim
that $U,\, TU,\, \dots, T^{q-1}U$ are pairwise disjoint.
To prove the claim, we have by the preceding,
 $F-F\circ T=1_U-\theta$ and $F\circ T-F\circ T^2=1_{TU}-\theta$.
Also since $f\circ T^2=\lambda^2 f$, it follows that $F-F\circ
T^2=1_V-2\theta$ for some clopen $V$.
However, $F-F\circ T^2=(F-F\circ T)+(F\circ T-F\circ
T^2)=1_U+1_{TU}-2\theta$, or $1_U+1_{TU}=1_V$; this  forces
$U\cap TU=\emptyset$.
The same reasoning may be applied to any pair of iterates of $T^iU$,
$0\leq i\leq q-1$, and  the claim follows.

By the claim and the fact that for any invariant measure $\mu$, $\mu(T^i
U)=1/q$, we have $\mu(\cup_{i=0}^{q-1} T^i U)=1$. Since
sufficiently many measures exist,  $X=\cup_{i=0}^{q-1} T^i U$.
 As $1_{T^i U}$
are all equivalent in $K^0(X,T)$, it follows that $q[1_U] = [X]$.

If $p \neq 1$, there exists a positive integer $s$ \st $ps \equiv 1
\mod
q$. Then we can apply the previous to $T^{s}$ (we no longer need $T^s$ to
be topologically transitive, since we have already constructed a suitable
logarithm of $f$).

\end{proof}

\bigskip

\begin{lem}\label{intco}
Let $(X,\,T)$ be a topologically transitive Cantor system and  $\theta\in E(X,\,T)$, $0<\theta<1$.  If $f,g$  are two continuous integer-valued functions on $X$ \st  $f-\theta\pmb 1$ and $g-\theta\pmb 1$ are real valued coboundaries, then  $f-g$ is an integer valued coboundary.
\end{lem}
\begin{proof}
Let $\lambda=\exp{2\pi i \theta}$. We have $F=f-\theta\cdot 1+F\circ T$ and  $G=g-\theta\cdot 1+G\circ T$, $F,G\in C(X,\,\R)$ . As $f$ is integer-valued, this yields $$\exp{(2\pi i F)}=\lambda^{-1}\exp{(2\pi iF\circ T)}=\lambda^{-1}\exp{(2\pi i F)}\circ T.$$ Hence $\exp{(2\pi i F)}$ and (with similar arguments) $\exp{(2\pi i G)}$ are two eigenfunctions for $T$ with the same eigenvalue $\lambda$.  As $T$ is transitive, they are proportional; hence  there exists $t\in [0,1)$ \st $\exp{(2\pi i F)}=\exp{(2\pi i t)}\cdot \exp{(2\pi i G)}$. Therefore, $h:=F-G-t\pmb 1\in C(X,\,\Z)$ and  $$\partial_Th=\partial_TF-\partial_TG=\partial_T(F-G)=f-g$$ which proves the lemma.

\end{proof}
\medskip

For $x\in\R$, let us denote by  $\lfloor x\rfloor$ (respectively $\lceil x\rceil$)  the largest (respectively, smallest) integer less  (respectively larger) than $x$ and recall that $\{x\}$ is the fractional part of $x$.
\medskip

\begin{cor}\label{cor1}
Let  $(X,\,T)$ be a Cantor minimal system. For each $\theta\in E(X,\,T)$,  let $\Theta(\theta)$ be defined by 
$$\Theta(\theta)=\left\{
\begin{tabular}{lr}
$\lfloor\theta\rfloor[1_X]+[1_{U_{\{\theta\}}}]$ & {\rm if }\ $ \theta\geq 0$, \\
\vspace{0.2cm}\\
$\lceil\theta\rceil[1_X]+[1_{U_{\{\theta\}}}]$ & {\rm if}\ $\theta<0$.
\end{tabular}%
\right.$$
Then  $ \Theta: E(X,\,T)\rightarrow K^0(X,\,T)$ is  an injective homomorphism.
\end{cor}

\begin{proof}
By Theorem \ref{clopen} and Lemma \ref{intco}, $\Theta$ is well-defined. 
Now  we prove that $\Theta:E(X,\,T)^+\rightarrow K^0(X,\,T)^+$ is a semigroup homomorphism. Recall first that if $\theta_1$ and $\theta_2$ are positive, then $\theta_1+\theta_2=\lfloor\theta_1+\theta_2\rfloor+\{\theta_1+\theta_2\}$. So if $0<\{\theta_1\}+\{\theta_2\}<1$, then $\{\theta_1+\theta_2\}=\{\theta_1\}+\{\theta_2\}$ and $\lfloor\theta_1+\theta_2\rfloor=\lfloor\theta_1\rfloor+\lfloor\theta_2\rfloor$. And if $1\leq \{\theta_1\}+\{\theta_2\}<2$ then $\lfloor\theta_1+\theta_2\rfloor=\lfloor\theta_1\rfloor+\lfloor\theta_2\rfloor+1$ and $\{\theta_1+\theta_2\}=\{\theta_1\}+\{\theta_2\}-1$.

\medskip

Therefore, to check that $\Theta:E(X,\,T)^+\rightarrow K^0(X,\,T)^+$ is a semigroup homomorphism, it suffices to consider the following cases.

\bigskip
\begin{enumerate}
\item {\it If  $0\leq\theta_1+\theta_2<1$}. By Theorem \ref{clopen}, there exist  clopen sets $U_{\theta_1}, U_{\theta_2}, U_{\theta_1+\theta_2}$ and continuous real valued functions $F_1, F_2$ and $F$ \st $1_{U_{\theta_1+\theta_2}}=F-F\circ T +(\theta_1+\theta_2)1_X$ and $1_{U_{\theta_i}}=F_i-F_i\circ T$, $i=1,2$. So
$$1_{U_{\theta_1+\theta_2}}-(1_{U_{\theta_1}}+1_{U_{\theta_2}})=(F-F_1-F_2)-(F-F_1-F_2)\circ T$$
and therefore, $1_{U_{\theta_1+\theta_2}}-(1_{U_{\theta_1}}+1_{U_{\theta_2}})\in (\pmb 1-T)C(X,\,\R)$.  So by Lemma \ref{intco},
$[1_{U_{\theta_1+\theta_2}}]=[1_{U_{\theta_1}}]+[1_{U_{\theta_2}}]$ in $K^0(X,\,T)$.
\bigskip

\item {\it If $1\leq\theta_1+\theta_2< 2$}. Set $\theta=\theta_1+\theta_2-1$.
By Theorem \ref{clopen}, there exists clopen sets $U_{\theta}, U_{\theta_j}$, real valued functions $F_j, j=1,2$ and $F$ \st for $j=1,2$, $1_{U_{\theta_j}}-\theta_j\pmb 1=F_j-F_j\circ T$ and $1_{U_{\theta}}-\theta\pmb 1=F-F\circ T$.
Then $f_j=\exp(2\pi i F_j)$, $j=1,2$ and $f=\exp(2\pi i F)$ are eigenfunctions with respect to $\theta_j$ and $\theta$.
As $f_1f_2f^{-1}$ is eigenfunction with respect to the eigenvalue $1$, there exists $t\in\R$ \st $\exp(2\pi i t)f_1f_2f^{-1}=1$ and therefore, $F_1+F_2-F-t\in C(X,\,\Z)$. Then
$$1_{U_{\theta_1}}
+1_{U_{\theta_2}}-1_{U_{\theta}}- 1_X=(\pmb 1 -T)(F_1+F_2-F-t\pmb 1)=(\pmb 1-T)(F_1+F_2-F).$$
Hence, $$[1_{U_{\theta_1}}]+[1_{U_{\theta_2}}]=[1_{U_\theta}+1]=[1_X]+[1_{U_\theta}].$$
\end{enumerate}
\bigskip

Therefore, if $0\leq \{\theta_1\}+\{\theta_2\}<1$, then $\{\theta_1+\theta_2\}=\{\theta_1\}+\{\theta_2\}$ and $\lfloor\theta_1+\theta_2\rfloor=\lfloor\theta_1\rfloor+\lfloor\theta_2\rfloor$ and by case 1) we have
\begin{equation}\label{*}
\Theta(\theta_1+\theta_2)=\Theta(\theta_1)+\Theta(\theta_2).
\end{equation}
If $1\leq \{\theta_1\}+\{\theta_2\}<2$, then $\lfloor\theta_1+\theta_2\rfloor=\lfloor\theta_1\rfloor+\lfloor\theta_2\rfloor+1$ and $\{\theta_1+\theta_2\}=\{\theta_1\}+\{\theta_2\}-1$. Then by case 2), we will get (\ref{*}).
\bigskip

Moreover, $\Theta(0)=[0]$. So we have a semigroup homomorphism  which can be extended to a group homomorphism $\Theta:E(X,\,T)\rightarrow K^0(X,\,T)$.   Theorem \ref{clopen} implies that the kernel of $\Theta$ is trivial.
\end{proof}

\bigskip
If $(X,\,T)$ is a  Cantor minimal system   and $(G,\,u)=(K^0(X,\,T),[1_X])$, let us write (see Section 2,3) to simplify notation, $$J(X,\,T)=J(G,\,u),\ \ \ \ I(X,\,T)=I(G,\,u).$$
Note that  the group $I(X,\,T)$ is denoted  by $I(G,\,G^+,\,u)$  in \cite{cortez}.
\medskip

Then by Section 2.3, Theorem \ref{clopen} and Corollary \ref{cor1}, $E(X,\,T)$ is a subgroup of ${\rm Im}(\Phi)=I(X,\,T)$ and $\Theta(E(X,\,T))\subset J(X,\,T)$. Denoting $\Phi^{-1}(E(X,\,T))$ by $H(X,\,T)$, we have the following corollary.
\bigskip

\bigskip

\begin{cor}\label{cor3}
Suppose $(X,\,T)$ is a Cantor minimal system. Then the short exact sequence of abelian groups, $$\Inf(K^0(X,\,T)){\longrightarrow} H(X,\,T)\stackrel{\Phi}{\longrightarrow}E(X,\,T)$$ splits (positively).
\end{cor}
\bigskip
The following proposition, already stated in  \cite{GPS} and proved in  \cite{Ormes},  is a consequence of Corollary \ref{rational} and  Lemma \ref{intco}.
\bigskip

\begin{prop}\label{cor2}
Let   $(X,\,T)$ be a Cantor minimal system and let $\Theta$ be the injective homomorphism defined in Corollary \ref{cor1}. Then  $$\Theta(E(X,\,T)\cap\Q)=\Q(K^0(X,\,T),\,[1_X]).$$
\end{prop}
\begin{proof}
 By Corollary \ref{rational}, if $\theta=p/q\in E(X,\,T)$,  then there exists a  clopen set $U$ \st $q[1_U]=p[1_X]$. So from  the definition of  $\Q(K^0(X,\,T),\,[1_X])$, the result   follows  from Corollary \ref{cor1}.
 \end{proof}
\bigskip

\begin{rem} \label{rat}
For a Cantor minimal system $(X,T)$, the rational spectrum, $E(X,\,T)\cap \Q$ is therefore invariant under strong orbit equivalence. This was already observed in \cite{GPS}, \cite{Ormes}.
\end{rem}
\bigskip

\begin{thm}\label{torsion}
Let $(X,T)$ be a Cantor minimal system and $\Theta$ be as in  Corollary \ref{cor1}. Then $K^0(X,\,T)/\Theta(E(X,\,T))$ is torsion-free.
\end{thm}
\begin{proof}
Let $g\in K^0(X,\,T)$. Let us show that if $ng\in \Theta(E(X,\,T))$ for some $n\in \N$ then $g$ itself belongs to $\Theta(E(X,\,T))$.
Since $\Theta(E(X,\,T))$ is totally ordered we can assume that $ng$ is positive (otherwise, replace it by $-ng$) and as $K^0(X,\,T)$ is unperforated, $g\in K^0(X,\,T)^+\setminus \{0\}$. By the assumption, there exists $\theta\in E(X,\,T)\cap \R^+$ \st $ng=\Theta(\theta)$. 
\bigskip

If $\theta\in\Q$, then by Proposition \ref{cor2}, $ng\in \Q(K^0(X,\,T),\,[1_X])$, hence $g\in\Q(K^0(X,\,T),\,[1_X])$. By Proposition \ref{cor2} again, $g\in \Theta(E(X,\,T))$.

\bigskip

 We can therefore assume $\theta\notin\Q$ and that there exists a proper clopen set $U$ of $X$ \st  $ng=\lfloor\theta\rfloor[1_X]+[1_{U}]$ for  some clopen subset $U$ \st
$1_{U}-\{\theta\}  1_X$ is a real coboundary. 
\bigskip

Suppose first that $n>\lfloor\theta\rfloor$. Then $$ng=\lfloor\theta\rfloor[1_X]+[1_U]\leq (n-1)[1_X]+[1_U]\leq n[1_X].$$ So, as $K^0(X,\,T)$ is unperforated, $0<g\leq [1_X]$. 
Moreover, $ng\neq n[1_X]$ which yields that $g\neq [1_X]$ and by \cite[Lemma 2.5]{GW}, there exists a proper clopen set $V$ of $X$ \st $g=[1_V]$.
\bigskip

Thus $n1_V-\lfloor\theta\rfloor1_X-1_U=G-G\circ T$, for some $G\in C(X,\,\Z)$. As $\{\theta\}1_X-1_U=F-F\circ T$ for some $F\in C(X,\,\R)$, we have 
$$1_V-{\theta\over n}1_X=1_V-{{\lfloor\theta\rfloor+\{\theta\}}\over n}1_X={F+G\over n}-{F+G\over n}\circ T$$
which by Theorem \ref{clopen} means that $\theta/n\in E(X,\,T)$ and $$\Theta(\theta/n)=[1_V]=g.$$

If $n\leq \lfloor\theta\rfloor$, write $\lfloor\theta\rfloor=k n+s$ with $0\leq s<n$ and $k\geq 1$. As $ng=(kn+s)[1_X]+[1_U]$, we have 
$$n(g-k[1_X])=s[1_X]+[1_U].$$
In particular, $g-k[1_X]\geq 0$ and since $s<n$, the previous case implies that $g-k[1_X]\in \Theta(E(X,\,T))$.
\end{proof}

\bigskip

Let $(X,\,T)$ be a Cantor minimal system with $\Inf(K^0(X,\,T))=\{0\}$. Maintaining our notation, $J(X,\,T)\cong I(X,\,T)$;  by Corollary  \ref{exact}, we can identify $I(X,\,T)/E(X,\,T)$ with $J(X,\,T)/\Theta(E(X,\,T))$.

\begin{cor}[\cite{cortez}, Theorem 1]\label{cort}
Let $(X,\,T)$ be a Cantor minimal system with $\Inf(K^0(X,\,T))=\{0\}$.

Then the quotient group $I(X,\,T)/E(X,\,T)$ is torsion-free.
\end{cor}
\begin{proof}
Let $g\in K^0(X,\,T)$. If for some $n\in\N$, there exists $\theta\in E(X,\,T)$ with $ng=\Theta(\theta)$, then for every $\tau\in S(K^0(X,\,T),\,[1_X])$,
$$n\tau(g)=\tau(ng)=\tau(\Theta(\theta))=\theta.$$
Hence $g\in J(X,\,T)$;  so the torsion elements of $K^0(X,\,T)/\Theta(E(X,\,T))$ and of $J(X,\,T)/\Theta(E(X,\,T))$ are the same. By Theorem \ref{torsion} and the above identification, the corollary is proved.
\end{proof}

\bigskip

\begin{rem}
For the identification between $J(X,\,T)/\Theta(E(X,\,T))$ and $I(X,\,T)/E(X,\,T)$ to be true, $\Inf(K^0(X,\,T))$ has to be trivial. In \cite{cortez}, an example of a Cantor minimal system for which neither $\Inf(K^0(X,\,T))$ nor the torsion subgroup of $I(X,\,T)/E(X,\,T)$ are  trivial, is described. 
\end{rem}

\section{Irrational miscibility } 

Let $G$ be a simple dimension group. With  notation and results of Section 2, recall that we have for $u\in G^+\setminus \{0\}$, a short exact sequence
$$0\rightarrow \Inf(G)\longrightarrow J(G,\,u)\stackrel {\Phi}{\longrightarrow }I(G,\,u)\longrightarrow 0$$
where $J(G,\,u)=\{g\in G:\ \hat{g} \ {\rm is\ constant}\ \}$ and $\Phi(g)=\hat{g}$.
\begin{defn} Let $(G,\,u)$ be a non-cyclic simple dimension group.
 We
say that
$(G,u)$ is {\it irrationally miscible } if $\Phi(J(G,\,u))\subseteq \Q$. Moreover,  $G$ is {\it globally irrationally miscible \/} if
for all choices of order units $u$, $(G,u)$ is irrationally miscible .
\end{defn}

The following is an immediate consequence. 

\begin{prop}\label{miscible }
Let $(G,u)$ be a non-cyclic simple dimension group with order unit $u$. Then $(G,\,u)$  is irrationally miscible  iff $\bigcap_{S(G,u)} \tau(G)\subset\Q$.
\end{prop}

\bigskip

\begin{rem}In general, we cannot take the intersection over the set of pure traces, $\bigcap_{\partial_e
S(G,u)} \tau(G)$,  as shown by the following  drastic example.

Let $K$ be the set of real algebraic
numbers; this is a countable subfield of the reals. Equipped with the sums
of squares ordering, this is well-known (and easily proved) to be a simple
dimension group and a partially ordered ring with $1$ as order unit, and
its pure traces with respect to the order unit $1$ are given by $r \mapsto
\gamma(r)$ where $\gamma: K \rightarrow K$ is a Galois automorphism (we
view the second copy of $K$ as a subgroup of the reals).

In particular, for every pure trace $\gamma$ on $(K,1)$, we have
$\gamma(K) = K$ (viewing the latter as a subgroup of $\R$), and thus
$\cap_{\gamma \in \partial_e S(G,u)} \gamma(K) = K$.
However,
$$
\bigcap_{\tau\in {S(K,1)}} \tau(K) = \Q.
$$
It suffices to show $\Set{\lambda \in \R}{\hat g = \lambda \pmb 1} = \Q$.
Select $g \in K$; if $\hat g = \lambda \pmb 1$, then for every Galois
automorphism, we have that $\gamma(g)$ is the same constant, $\lambda$;
hence $g$ belongs to the fixed point subgroup of the Galois group, hence
is a rational number.

If we replace the order unit $1$ by any other order unit, the same
conclusions apply, as is easy to verify, since $K$ is a field. Hence $K$
with the sums of squares ordering is globally irrationally miscible .

Of course, we can obtain a simpler example with  quadratic
squares ordering. There are two pure traces, each has range $K$ itself,
but the intersection over all the traces of their images is just $\Q$, via
the same argument. The infinite-dimensional example is more interesting.
\end{rem}

\bigskip

Before giving sufficient conditions for a non-cyclic simple dimension group to be globally irrationally miscible , let us describe the implication of this property for  dynamical systems.
\bigskip

Let $(X,\,T)$ be a Cantor minimal system and $(G,\,u)=(K^0(X,\,T),\,[1_X])$. By  results in section 3, we have 
$$\Theta(E(X,\,T))\subset J(X,\,T)\ \ \ {\rm and}\ \ \ \Phi\circ \Theta(\theta)=\theta,\ \forall \theta\in E(X,\,T).$$
The condition $\cap_{\tau \in S(G,u)} \tau(G) \subset
\Q$ is not affected when we factor out the infinitesimals from $G$.
Therefore, if $(K^0(X,\,T),\,[1_X])$ is irrationally miscible , then the additive group of continuous eigenvalues of any Cantor minimal systems  orbit equivalent to $(X,\,T)$ is  contained in $\Q$.

\bigskip

\begin{prop}\label{unique trace} If $G$ is a simple dimension group with a rational-valued trace, then $G$ is
globally irrationally miscible .\end{prop}
\begin{proof}
Follows immediately from Proposition \ref{miscible } above, since
there exists a trace
$\tau$ \st  $\tau(G)\subset\Q$, and renormalizations  (changing the order unit)  do not affect this
property.
\end{proof}

Recall that a matrix has {\it equal column
sums} if all its column sums are the same.

\begin{cor}\label{circulant}

Let $\{A_n \}$ be a sequence of nonnegative integer matrices, with equal column sum \st
the direct limit $G$ of $ A_n:\Z^{f(n)} \rightarrow \Z^{f(n+1)}$ is simple (and non cyclic). Then $G$ is globally irrationally miscible .
\end{cor}

\begin{proof}
Let $c_n$ be the column sum of $A_n$. For $a = [a(n), n] \in G$ and $c(n) = \prod_{l \le n} c_l$ then $\tau(a) = \frac{c(n)}{\sum_n a(n)}$ defines a trace on $G$, and $\tau (G) \subset \Q $.
\end{proof}


\begin{eg}
The condition of Corollary \ref{circulant}  applies when all the matrices  $\{A_n \}$ are circulant. For example, let $(B,\,V)$ be the Bratteli diagram with $|V_n|=2$ for all $n\geq 2$  and for any $n\geq 1$, $A_n$  be the symmetric matrix, $\begin{bmatrix} k_n& \ell_n\\\ell_n& k_n \end{bmatrix}$.  By Corollary \ref{circulant}, the associated dimension group, $G$ is irrationally miscible . Moreover,  $\Inf(G)$ is trivial. Therefore any Cantor minimal system $(X,\,T)$ with $K^0(X, T)$ order isomorphic to $G$ has only rational continuous eigenvalues.


\end{eg}

\bigskip

Recall that for an abelian group $H$, $\rk H$ denotes the dimension of the rational vector space $H \otimes \Q$.

\medskip

\begin{prop}\label{rank}
Let $G$ be a simple dimension group with $n$ pure traces, $n>1$.

If  either $\rk G/\Inf(G) = n$ or both $\rk G/\Inf(G) = n+1$ and $G/\Inf(G)$ is finitely
generated, then $G$ is globally irrationally miscible .
\end{prop}

\begin{proof}
Let $u$ be an order unit of $G$ and $\{\sigma_i:\ 1\leq i\leq n\}$ be the set of pure (normalized) traces on $G$. Let $\hat{G}$ denote both the (simple dimesnion) group $G/\Inf(G)$ and the dense image of $G$ in $\Aff(S(G,\,u))\iso\R^n$. 

Let $\hat{\Phi}$ denote the restriction to $\hat{G}$ of the map $\Phi$ from $\Aff(S(G,\,u))$ to $\R^{n-1}$, given by
$$\Phi(f)=(\sigma_2(f)-\sigma_1(f),\,\cdots,\,\sigma_n(f)-\sigma_1(f)),\ f\in \Aff(S(G,\,u)).$$
If $(G,\,u)$ is not irrationally miscible , then there exists an irrational number $\lambda$ and $g\in G$ \st $\hat{g}=\lambda\hat{u}$. Then the $\ker\hat{\Phi}$ contains both $\hat{g}$ and $\hat{u}$, so has rank at least $2$ and thus $Im(\hat{\Phi})$ has rank at most $\rk(\hat{G})-2$. 

As $\Phi$ is continuous and surjective, $\hat{\Phi}(\hat{G})$ is dense in $\R^{n-1}$ and therefore must be of rank at least $n-1$. Hence, we obtain a contradiction if $\rk(\hat{G})=n$.

If $\rk(\hat{G})=n+1$ and $\hat{G}$ is finitely generated, then we note that a free dense subgroup of $\R^{n-1}$ must have rank at least $n$, again reaching a contradiction.
\end{proof}
\bigskip

If $G$ is a simple dimension group with an order unit $u$, we denote (as in the proof of the above proposition) by $\hat{G}$ both the (simple dimesnion) group $G/\Inf(G)$ and the dense image of $G$ in $\Aff(S(G,\,u))\iso\R^n$.  Set $J=\{\hat{g}\in\hat{G}:\ \hat{g}\  {\rm is\ constant}\}$. 

\bigskip

\begin{prop}\label{sprank}
Let $G $ be a simple dimension group  with
order unit $u$ \st $\rk(G/\Inf G)<\infty$.  Let $l$ be  the number of pure traces of $(G,\,u)$. Then $l\leq \rk(\hat{G})$ and $\rk (\hat{G}/J) \geq l-1$. If $G$ itself has finite rank, then
$$
\rk G \geq \rk J + l-1;
$$
and if $G/\Inf G$ is finitely generated, then $\rk G/J \geq l$.
\end{prop}
\begin{proof}
Since $\Inf G \subset J$, $G/J \iso (G/\Inf G)/(J/\Inf G)$, whence
the former is of finite rank.  Let $\{\sigma_1, \cdots, \sigma_l\}$ be the set of pure traces of $G$ and let $\Phi:\Aff(S(G,\,u))\rightarrow\R^{l-1}$ be given (as in th eproof of Proposition \ref{rank}) by $$\Phi(f)=(\sigma_2-\sigma_1(f),\,\cdots,\,\sigma_l-\sigma_1(f)),\ f\in\Aff(S(G,\,u)).$$
Then $\Phi(G)$ is a dense subgroup of $\R^{l-1}$, and as $J\subset\ker(\Phi)$, $\Phi(G/J)$ is isomorphic to a dense subgroup of $\R^{l-1}$. Hence, $\rk(G/J)\geq l-1$. If $G/\Inf(G)$ is finitely generated, then $\rk(G/J)\geq l$ (as a dense subgroup of $\R^{l-1}$ must have rank at least $l$).

\end{proof}

\bigskip

If $(X,\,T)$ is a Cantor minimal  system and 
 $(G,\,u)=(K^0(X,\,T),\,[1_X])$,   then by results in section 3,  $H(X,\,T)=\Theta(E(X,\,T))\subset J(X,\,T)$ and therefore, 
 $$\rk(E(X,\,T))\leq \rk(J(X,\,T)).$$ If $(X,\,T)$ has $n$ ergodic invariant measures, then the proposition above  yields 
 \begin{equation}\label{dur}
 \rk(K^0(X,\,T))\geq \rk(E(X,\,T))+n-1
 \end{equation}
or $$\rk(K^0(X,\,T))\geq \rk(E(X,\,T))+n$$ 
if $K^0(X,\,T)$ is finitely generated. 

\bigskip

\section{Weak mixing  Bratteli-Vershik systems }

A topological dynamical system $(X,\,T)$ is {\it weakly mixing\/} if the product system
$(X\times X,\,T\times T)$ is topologically transitive, or equivalently, if  for any two nonempty open
sets    $U$ and $V$ of $X$, the set $\mathcal N(U,\,V):=\{n\in\mathbb
Z:\ \  T^nU\cap V\neq \emptyset\}$ is {\it thick\/} \cite{furs}. Recall that a subset of $\Z$ is {\it thick\/} if $$\forall \ k\in\mathbb N,\ \ \exists\ n;\ \ n, n+1, \cdots, n+k\in
\mathcal N(U,\,V).$$
By \cite[Proposition 3]{indian}, $(X,\,T)$ is weakly mixing if and only if for any pair of nonempty open sets $U$ and $V$ there exist $n\in\mathbb N$ \st
$n,n+1\in{\mathcal N}(U,\,V)
$.

Recall that for a minimal system, $(X,\,T)$,  to be weakly mixing, it is necessary and sufficient that its additive continuous spectrum 
 $E(X,\,T)$ is equal to $\Z$ (see \cite{glasner} for example).
 \medskip
 
 We will denote by $\mathcal  W\mathcal M$, the collection of weakly mixing Cantor minimal systems. 

 \medskip

For a simple dimension group $(G,\,u)$, let  SOE$(G,\,u)$ denote the class of Cantor minimal systems $(X,\,T)$ \st $(K^0(X,\,T),\,[1_X])$  is order isomorphic to $(G,\,u)$.
 Recall (see Theorem  \ref{storbit})  that two Cantor minimal systems belong to SOE$(G,\,u)$ if and only if they are strongly orbit equivalent.
 
 Similarly, for a simple dimension group $(G,\,u)$ whose infinitesimals subgroup is trivial, let OE$(G,\,u)$ be the class of all Cantor minimal systems $(X,\,T)$ \st  $(K^0(X,\,T)/\Inf(K^0(X,\,T)),\,[1_X])$ is order isomorphic to $(G,\,u)$. Two Cantor minimal systems belong to OE$(G,\,u)$ if and only if they are orbit equivalent (see Theorem \ref{orbit}).
 \bigskip

\begin{prop}\label{weakmixing}
Let $(G,u)$ be a simple dimension group with order unit $u$. If $\Q(G,\,u)=\Z$ and
$(G,\,u)$ is irrationally miscible , then SOE$(G,\,u)\subset {WM}.$
\end{prop}

Remark. In other words, if $(X,\,T)$ is a minimal Cantor system \st $(K^0(X,\,T), [1_X])$ is irrationally miscible  and has trivial {\it rational}  spectrum, then $(X,\,T)$ is weakly mixing.

\begin{proof}
Let $(X,\,T)$ be a Cantor minimal system with $(K^0(X,T), [1_X])=(G,\,u)$. As $(G,\,u)$ is irrationally miscible  and by Theorem \ref{clopen}, $E(X,\,T)\subseteq \Q$. So by Proposition \ref{cor2}, $\Theta(E(X,\,T))=\Q(G,\,u)=\Z$.
\end{proof}

\bigskip

The following theorem is a direct corollary of  \cite[Theorem 6.1]{Ormes}.

\begin{thm}
Let $(G,\,u)$ be a simple dimension group whose rational subgroup $\Q(G,\,u)$ is equal to $\Z$. Then there exists a topologically weakly mixing Cantor minimal system $(X,\,T)$ belonging to SOE$(G,\,u)$.
\end{thm}

\begin{proof}
Let $(Y,\,\nu,\,T)$  be a measurable weakly mixing system on a Lebesgue space. By \cite[Theorem 6.1]{Ormes}, there exists a Cantor minimal system $(X,\,S)$ with $(K^0(X,\,S),\,[1_X])\simeq (G,\,u)$ and $\mu\in{\mathcal M}_S(X)$ \st  $(X,\,S,\,\mu)$ is isomorphic to $(Y,\,\nu,\,T)$. In particular, $E(X,\,S)=\Z$. 
\end{proof}

\bigskip


\begin{prop}\label{last}
Let  $(G,u)$ be a simple non-cyclic dimension group with trivial infinitesimal subgroup.  If $(G,\,u)$ is irrationally miscible  and $\Q(G,\,u)\subset\Z$ then $${\rm OE}(G,\,u)\subset {\mathcal W\mathcal M}.$$
\end{prop}

\begin{proof}
Recall that by Proposition \ref{miscible }, if $(H,\,v)$ is a simple non-cyclic dimension group, then $(H,\,v)$ is irrationally miscible  if and only if $(H/\Inf(H),\,[v])$ is also irrationally miscible . Moreover, if $x\in\Q(H,\,v)$ then the image $[x]$ in $H/\Inf(H)$ belongs to $\Q(H/\Inf(H),\,[v])$. 

Let $(X,\,T)$ be a Cantor minimal system belonging to ${\rm OE}(G,\,u)$. Then $(K^0(X,\,T),\,[1_X])$ is irrationally miscible  and $$\Theta(E(X,\,T))\subset\Q(K^0(X,\,T),\,[1_X]).$$ 
As $\Q(K^0(X,\,T),\,[1_X])=\Z$, it follows that $(X,\,T)$ is weakly mixing.
\end{proof}
\bigskip

 In the rest of this section, we are going to prove that:

\begin{thm}\label{propo}
Let $B=(V,\,E)$ be a simple Bratteli diagram  \st $\Q(K^0(V,\,E),\,[v_0])=\Z$. Then 
 there exists a telescoping $\tilde{B}=(\tilde{V},\,\tilde{E})$ of $B$ and a proper 
ordering on $\tilde{B}$ whose associated Bratteli-Vershik system
  is weakly mixing.
\end{thm}
To prove this theorem, we will use Corollary \ref{gcd} and Lemma  \ref{new}. The first one is a result of \cite{Ormes} which we prove for   sake of completeness. It gives a characterization of Bratteli diagrams whose associated dimension group has a trivial  rational subgroup. Lemma \ref{new} gives a technical characterization for a Bratteli-Vershik transformation to be weakly mixing.

\medskip

\begin{lem}\label{weak}
Let $(G,\,u)$ be a simple dimension group with order unit $u$. Then  its rational subgroup $\Q(G,\,u)$ is trivial if and only if 
 The equation $nx=u$ is not solvable for any  $x\in G$, $x\neq u$, and $n\in\mathbb Z$.
\end{lem}
\begin{proof}
If $\Q(G,\,u)=\Z$, then (1) is clearly satisfied. By Proposition \ref{unique trace}, we only have to prove that (1) implies $\Q(G,\,u)=\Z$.

If not, let $x\in G$ \st $px=qu$ with $(p,\,q)=1$. As there exists $a, b\in\Z$ with $ap+bq=1$ we have $bpx=bqu=(1-ap)u$. Hence, $p(bx+au)=u$, which contradicts (1).
\end{proof}

\bigskip
  Lemma \ref{weak} forces a combinatorial property  for any Bratteli diagram $B=(V,\,E)$ with $K^0(V,\,B)=G$. That is   Corollary \ref{gcd}.
\bigskip

Let $B=(V,\,E)$ be a Bratteli diagram with 
$$V=\sqcup_{k\geq 0} V_k,\ V_0=\{v^0\},\ V_k=\{v_1^k,\,v_2^k,\dots,\,v_{n(k)}^k\}.$$
For $k\geq 1$ and $1\leq j\leq n(k)$, let $h_i^k$ denote the number of (finite) paths from $v^0$ to $v_j^k$. We assume that $h_k=(h_1^k,\,\dots,\,h_{n(k)}^k)\in\mathbb N^{n(k)}$.

\begin{cor}\label{gcd}
Let $B=(V,\,E)$ be as above and $K^0(V,\,E)$ its associated (simple) dimension group. Then  $\Q(K^0(V,\,E),\,[v^0])=\Z$ if and only if for all $k\geq 1$, ${\rm gcd}(h_1^k,\,\dots, h_{n(k)}^k)=1$.
\end{cor}

\bigskip

\begin{lem}\label{new}
Let $(B,\,\leq)$ be a properly ordered Bratteli diagram \st the minimal path $e_{\rm min}=(e_1,\,e_2,\,\dots)$ satisfies $$r(e_k)=v_1^k,\ \ \forall k\geq 1.$$
Then the associated Bratteli-Vershik system $(X_B,\,T_B)$ is weakly mixing if and only if the following  condition $C$(\ref{c}) is satisfied:
$$
\forall\ k\geq 1, \ \exists n; \ \{n,n+1\}\subset{\mathcal N}(C(e^k),\,C(e^k))\eqno{C(4.7)}\label{c}
$$
where $C(e^k)$ is the cylinder set defined by the finite path $e^k=(e_1,\,\dots,\,e_k)$.
\end{lem}
\begin{proof}
By \cite[Proposition 3]{indian}  condition $C($\ref{c}) is necessary. We prove its sufficiency in two steps.
\bigskip

{\noindent} {\bf Step 1:} Condition $C$(\ref{c}) implies that
$$
\forall\ k\geq 1, \forall \ a, b, \exists n; \ \{n,n+1\}\subset{\mathcal N}(C(a),\,C(b)),\eqno{C_1(4.7)}\label{c1}
$$
where $a, b$ are two finite paths from $v^0$ to $v_1^k$.
\bigskip

To see this, let us first notice that if $\{n, n+1\}\subset{\mathcal N}(C(e^k),\,C(e^k))$ and $0\leq i\leq h_1^k$, then 
$$\{n+i, n+i+1\}\subset{\mathcal N}(C(e^k),\,C(T^i(e^k))).$$
As for $a, b\in E(v_1^0,\,v_1^k)$, with $a<b$, there exists $0\leq i\leq h_1^k$ and $1\leq j\leq h_1^k-i$ \st $a=T_B^i(e^k)$ and $b=T_B^{j+i}(e^k)$, condition $C_1($\ref{c1}.7) follows.
\bigskip

\noindent{\bf Step 2:} Condition $C_1(\ref{c1}.7$) implies that $(X_B,\,T_B)$ is weakly mixing.
\medskip

As the cylinders form a basis for the topology on $X_B$,  it is enough to show that for any two cylinder sets $U$ and $V$, $\mathcal N(U,\,V)$ contains two consecutive integers. Then for $k$ large enough there exist  two finite paths $a, b\in E(v^0,\,v_1^k)$ \st  $C(a)\subset U$ and $C(b)\subset V$. As $\mathcal N(C(a),\,C(b))\subset \mathcal N(U,\,V)$, step 2 is proved.

\end{proof}

\bigskip

Let us recall that if $\left\{h_1, h_2, \dots, h_n\right\}$ is a set of $n$ positive integers having greatest common divisor one, then for every integer $m$ bigger than $({\rm min}(h_i)-1)({\rm max}(h_i)-1)$, there exist $x_1, x_2, \dots, x_n\in\Z^+$ \st $m=\sum_{j=1}^nx_jh_j$. This result was proved in 1935 by I. Schur, but not published until 1942 by A. Brauer \cite{brauer}. This  will be used in the proof of Proposition \ref{propo}. 
\bigskip

\noindent{\bf Proof of Theorem \ref{propo}.}

Let $B=(V,\,E)$ with $V_0=\{v^0\}$ and $V_k=\{v_1^k,\,\dots,\,v_{n(k)}^k\}$. By Lemma \ref{gcd} and Schur's result, there exist for each  $k\geq 1$, some $\ell_k\in\N$  and $x_1^k,\,\dots,\,x_{n(k)}^k\in\N\cup\{0\}$ \st 
\begin{equation}
\ell_kh_1^k+1=\sum_{j+1}^{n(k)} x_j^kh_j^k; \ x_j^k\geq 0. \label{recall}
\end{equation}
By induction, we can construct a sequence 
$$1=k_1<k_2<k_3<\cdots $$
\st 
\begin{eqnarray*}
\# e(v_1^{k_m} ,\,v_1^{k_{m+1}})&>&\ell_{k_m}+x_1^{k_m},\\
\# e(v_j^{k_m},\,v_1^{k_{m+1}})&>&x_j^{k_m},\ \ \  2\leq j\leq n(k_{m}).
\end{eqnarray*}
where for $p<q$, $\#e(v_j^p,\,v_i^q)$ is the number of finite paths from $v_j^p$ to $v_i^q$. 

So by telescoping the diagram along the sequence $({k_m})_{m\geq 1}$, we can assume that for every $k\geq 1$,  (\ref{recall}) holds and 
\begin{eqnarray*}
\# e(v_1^{k} ,\,v_1^{k+1})&>&\ell_{k}+x_1^{k},\\
\# e(v_j^{k},\,v_1^{k+1})&>&x_j^{k},\ \ \  2\leq j\leq n(k).
\end{eqnarray*}

To complete the proof of the theorem, we define an ordering $\leq$ on $B=(V,\,E)$ satisfying condition $C$(\ref{c}) and \st $(B,\,\leq)$ is properly ordered. For $k\geq 1$ and $2\leq j\leq n(k+1)$, we order the edges of $r^{-1}(v_j^{k+1})$ from left to right. To describe the linear order on $r^{-1}(v_1^k)$, let us introduce the following notation.

\bigskip
For $1\leq j\leq n(k)$, let ${\rm max}_i^k=\#e(v_i^k,\,v_1^{k+1})$ and 
$$e(v_i^k,\,v_1^{k+1})=\{e(i,\,j):\ 1\leq j\leq {\rm max}_i^k\}.$$

Then the order on $r^{-1}(v_1^{k+1})$ is determined by ordering the pairs $\{(i,\,j):\ 1\leq j\leq {\rm max}_i^k,\ 1\leq i\leq n(k)\}$ as follows:
 \begin{eqnarray*}
&(1,\,1)&<(1,\,2)<\dots <(1,\,\ell_k+x_1^k)<\\
&(2,\,1)&<(2,\,2)<\dots <\dots <(2,\,x_2^k)<\\
&\vdots&\\
&(n(k),\,1)&<(n(k),\,2)<\dots <(n(k),\,x_{n(k)}^k)<\\
&(1,\,\ell_k+x_1^k+1)&<\dots <(1,\,{\rm max}_1^k)<\\
&(2,\,x_2^k+1)&<\dots <(2,\,{\rm max}_2^k)<\\
&\vdots&\\
&(n(k),\,x_{n(k)}^k+1)&<\dots <(n(k),\,{\rm max}_{n(k)}^k).
\end{eqnarray*}

It is easy to check that with this ordering $(B,\,\leq)$ is properly ordered with a single min path (resp. max path) going through $v_1^k$ (resp. $v_{n(k)}^k$) for all $k\geq 1$.
\bigskip

To finish the proof, let us verify that condition $C$(\ref{c}) is satisfied.
Let $e=(e_1,\,e_2,\,\dots)$ be the minimal path of $X_B$. Then by the above definition of the ordering on $B$, $e_n\in e(v_1^{n-1},\,v_1^n)$ is equal to $e_n(1,\,1)$. So $e$ belongs to the cylinder set $C(e^k)$, where $e^k=(e_1,\,e_2,\,\dots,\,e_k)$. Since $\#e(v^0,\,v_1^k)=h_1^k$, 
$$T_B^{\ell_kh_1^k}(e)=(e_1,\,e_2,\,\dots,\,e_k,\,e_{k+1}(1,\,\ell_k+1),\,\dots)\in C(e^k)$$
and therefore, 
$$\ell_kh_1^k\in{\mathcal N}(C(e^k),\,C(e^k)).$$
Let $a=(e_1,\,\dots,\,e_k,\,e_{k+1}(1,\,\ell_k+1),\,\dots)\in C(e^k)$. By (\ref{recall}), $$T_B^{\ell_kh_1^k+1}(a)=(e_1,\,\dots,\,e_k,\,e_{k+1}(1,\,\ell_k+x_1^k+1),\,\dots)$$
also belongs to $C(e^k)$ and therefore, condition $C$(\ref{c}) is verified.

$\hfill\Box$
\bigskip

We adopt the following notation for Proposition \ref{trunc}.

\begin{itemize}
\item $\Sigma_k=\{ {\rm paths \ from} \ v_0\  {\rm to }  \ v_i^k\in V_k, i=1,\dots,n(k)\}$.
\item $\Sigma_k^{\Z}=$ shift space with the finite alphabet $\Sigma_k$.
\item $\pi_k: X_B\rightarrow \Sigma_k$ truncation map, restricts each infinite path  of $X_B$ to its first $k$ coordinates.
\item $Y_k=\{(\pi_k(T_B^nx))_{n\in\Z}:\ x\in X_B\}\subseteq\Sigma_k^\Z$.
\item $S_k$ is the shift map on $Y_k$.
\end{itemize}
It is well-known that $\pi_k$ is a factor map and $(X_B,\,T_B)$ is an extension of $(Y_k,\,S_k)$.

\begin{prop}\label{trunc}
Let $(B,\,\leq)=(V,\,E,\,\leq)$ be a properly ordered Bratteli diagram and 
$(X_B,\, T_B)$ the associated Bratteli-Vershik system. With the above notation, $(X_B,\,T_B)$ is weakly mixing if and only if  $(Y_k,\, S_k)$  is weakly mixing for all $k\geq 1$.
\end{prop}
\begin{proof}
If
$(X_B,\, T_B)\in {\mathcal W}{\mathcal M}$ then $(Y_k,\,S_k)\in {\mathcal W}{\mathcal {M}}$ as $\pi_k$ is a factor map.  Conversely,
it is enough to observe that for any two cylinder sets $U=[e_1,\cdots, e_k]$ and $V=[f_1, \cdots, f_k]$ of
$X_B$, $\mathcal N(U,\,V)$ is thick. Let us denote by $\tilde{U}$ (resp. $\tilde{V}$) the cylinder of $Y_k$ given by $\tilde{U}=\{y\in Y_k;\ y_0=(e_1,\,\dots,\,e_k)\}$ (resp. $\tilde{V}=\{y\in Y_k;\ y_0=(f_1,\,\dots,\,f_k)\}$). If $(Y_k,\,S_k)\in{\mathcal W}{\mathcal M}$ then $\mathcal {N}(\tilde{U},\,\tilde{V})$ is thick, and therefore $\mathcal {N}(U,\,V)$ has the same property. 
\end{proof}

\bigskip

One class of examples of simple dimension groups  satisfying  both the conditions of Proposition \ref{weakmixing} occurs  when we combine the gcd$=1$ condition with Corollary \ref{circulant}. 
\bigskip

 For example, let $G$ be an
extension,
$\Z \rightarrow G \rightarrow \Z[1/3]$ with the strict ordering induced by
the map to $\Z[1/3]$; an easy example  is the
stationary direct limit implemented by the matrix $\( \smallmatrix 2 & 1
\\ 1 & 2\\  \endsmallmatrix\)$; this yields a non-split extension. 
By choosing an element $u\in G^+$, not an integer multiple of any
other elements of $G$, $\Q(G,\,u)=\Z$. Moreover, the unique state $\tau$ of $G,\,u)$ satisfies $\tau(G)\subset \Q$. By Corollary \ref{circulant} and Proposition \ref{weakmixing}, SOE$(G,\,u)\subset {WM}.$

\medskip

On the other hand, $G/\Inf G $ is isomorphic to $\Z[1/3]$, so $(X,T)$ is
orbit equivalent to the $3$-odometer, which is not weakly mixing. Recall that by \cite[Corollary 2]{GPS},  any
uniquely ergodic Cantor minimal system is orbit equivalent to a minimal
Cantor system which is a Denjoy's or an odometer, hence is not weakly
mixing.
\medskip

However, when the system is not uniquely ergodic, somewhat more interesting
phenomena can occur and we may have a system which is not even orbit equivalent to a non-weakly mixing system. For instance, if  we consider $G$ to be the
 direct limit implemented by the matrices $\( \smallmatrix k_n & l_n
\\ l_n & k_n\\  \endsmallmatrix\)$, $n\geq 1$,  then $\Inf(G)$ is trivial and by Proposition  \ref{unique trace} there exists a rational valued trace on it. It is not hard to choose a sub-diagram which satisfies the gcd$=1$ condition as well. Then by Corollary \ref{last} any Cantor minimal system orbit equivalent to the Vershik system associated to this sub-diagram, is weakly mixing.

\bigskip


\end{document}